\documentclass[12pt]{article}
\usepackage[centertags]{amsmath}
\usepackage{amsfonts}
\usepackage{amssymb}
\usepackage{latexsym}
\usepackage{amsthm}
\usepackage{newlfont}
\usepackage{graphicx}
\usepackage{listings}
\usepackage{booktabs}
\usepackage{abstract}
\usepackage{enumerate}
\RequirePackage{srcltx}
\date{}

\newlength{\defbaselineskip}
\newcommand{\setlinespacing}[1]%
           {\setlength{\baselineskip}{#1 \defbaselineskip}}

\newcommand{\N}{{\mathbb{N}}}

\newcommand{\actaqed}{\hfill $\actabox$}
{\medskip\noindent \textit{Proof of #1. }}%
{\actaqed \medskip}

\def\D{{\mathcal D}}
\def\cD{{\mathcal D}}

\def\R{{\mathbb R}}

\def \<{\langle}
\def\>{\rangle}

\def \e{\varepsilon}

\def\ga{\gamma}

\def \sp{\operatorname{span}}

\newtheorem{Theorem}{Theorem}[section]
\newtheorem{Lemma}{Lemma}[section]

\newtheorem{Proposition}{Proposition}[section]
\newtheorem{Remark}{Remark}[section]

\newtheorem{Corollary}{Corollary}[section]
\numberwithin{equation}{section}

\newcommand{\be}{\begin{equation}}
\newcommand{\ee}{\end{equation}}

\begin{document}

\title{A remark on entropy numbers}
\author{V. Temlyakov}

\newcommand{\Addresses}{{% additional braces for segregating \footnotesize
  \bigskip
  \footnotesize

  V.N. Temlyakov, \textsc{University of South Carolina,\\ Steklov Institute of Mathematics,\\ Lomonosov Moscow State University, \\ and Moscow Center for Fundamental and Applied Mathematics
  \\
E-mail:} \texttt{temlyak@math.sc.edu}

}}

\maketitle
\begin{abstract}
{   Talagrand's fundamental result on the entropy numbers is slightly improved. Our proof uses different ideas based on results from greedy approximation.
  }
\end{abstract}

\section{Introduction}
\label{I}

The paper is motivated by recent successful applications of entropy numbers in the sampling discretization of integral norms of functions from finite dimensional subspaces. To discretize  the integral norms successfully, a  new technique was introduced.  This technique takes different forms in different papers but the common feature of its forms is the following. The new sampling discretization technique is a combination of probabilistic technique, in particular chaining technique, with results on the entropy numbers in the uniform norm (or its variants). Fundamental results from \cite{BLM}, \cite{Tal}, \cite{MSS} were used. The reader can find 
results on chaining in \cite{KoTe}, \cite{VTbook} and on generic chaining in \cite{Tal}.  We note that the idea of chaining goes back to the 1930s, when it was suggested by A.N. Kolmogorov. Later, these types of results have been developed in the study of the central limit theorem in probability theory (see, for instance, \cite{GZ}).
Also, the reader can find general results on metric entropy in \cite[Ch.15]{LGM},  \cite[Ch.3]{VTbook}, \cite[Ch.7]{VTbookMA}, \cite{Carl},
\cite{Schu} and in the recent papers \cite{VT156} and  \cite{HPV}.
Bounds for the entropy numbers of function classes are both important by themselves and by their important connections to other fundamental problems (see, for instance, \cite[Ch.3]{VTbook} and \cite[Ch.6]{DTU}).

Let $\Omega$ be a compact subset of $\R^d$ with the probability measure $\mu$. By $L_p$, $1\le p< \infty$, norm we understand
$$
\|f\|_p:=\|f\|_{L_p(\Omega,\mu)} := \left(\int_\Omega |f|^pd\mu\right)^{1/p}.
$$
By discretization of the $L_p$ norm we understand replacement of the measure $\mu$ by
a discrete measure $\mu_m$ with support on a set $\xi =\{\xi^j\}_{j=1}^m \subset \Omega$. This means that integration with respect to measure $\mu$ is replaced by an appropriate cubature formula. By $L_\infty(\Omega)$ we denote the space of continuous on $\Omega$ functions with 
the norm 
$$
\|f\|_\infty := \|f\|_{L_\infty(\Omega)} := \max_{x\in\Omega} |f(x)|.
$$

Let $X$ be a Banach space (or a linear space with the semi-norm $\|\cdot\|_X)$. For a set $W \subset X$ and a positive number $\e$ we define the covering number $N_\e(W,X)$
 as follows
$$
N_\e(W,X) :=\min \{n : \exists y^1,\dots,y^n :W\subseteq \cup_{j=1}^n B_X(y^j,\e)\},
$$
where $B_X(y,\e):=\{x\in X\,:\, \|x-y\|_X \le \e\}.$
It is convenient to consider along with the entropy $H_\e(W,X):= \log N_\e(W,X)$ (here and later $\log:=\log_2$) the entropy numbers $\e_k(W,X)$:
$$
\e_k(W,X) := \e_k(W,\|\cdot\|_X):=  \inf \{\e : \exists y^1,\dots ,y^{2^k} \in X : W \subseteq \cup_{j=1}
^{2^k} B_X(y^j,\e)\}.
$$
Denote  
$$
X^p_N := \{f:\, f\in X_N,\, \|f\|_p\le 1\}.
$$
It was understood in recent papers (see \cite{VT158}, \cite{VT159}, \cite{DPTT}, \cite{DPSTT1}, and \cite{DPSTT2}) that conditions on the entropy numbers $\e_k(X^p_N,L_\infty(\Omega))$ of the unit $L_p$-ball of the subspace $X_N$ in the uniform norm guarantee good results on the sampling discretization of the $\|f\|_p$ norms of $f\in X_N$. We note that behavior of the entropy numbers in the uniform norm is also important in some other problems. For instance,
it is known (see for instance \cite{VTbook}, section 3.6) that the problem of finding the right behavior (in the sense of order) of the entropy numbers of the unit balls of spaces of multivariate functions with mixed smoothness is equivalent to the Small Ball Problem from probability theory. We point out that the Small Ball Problem is not solved in dimensions $d\ge 3$. 

Further, in a recent paper \cite{Kos} it is was understood that in addition to the uniform norm the following weaker norm (semi-norm), which is popular in empirical processes, is useful in sampling discretization. Let a set $\Omega_n =\{x^j\}_{j=1}^n$ be a set of points from $\Omega$. Consider $L_\infty(\Omega_n)$ on $X_N$. The following result is from \cite{Kos} 

\begin{Lemma}\label{IL1} Let $p\in (2,\infty)$. Assume that for any $f\in X_N$ we have
\be\label{I1}
\|f\|_\infty \le M\|f\|_p
\ee
with some constant $M$. Then for $k\in [1,N]$ we have for any $\Omega_n$
\be\label{I2}
\e_k(X_N^p, L_\infty(\Omega_n)) \le C(p)M\left(\frac{\log n}{k}\right)^{1/p}  .
\ee
\end{Lemma}

We prove here (see Section \ref{B}) a slight improvement of the above lemma. We replace $\log n$ by $\log (2n/k)$. It is known that in the theory of Kolmogorov widths of finite dimensional unit $\ell_p$-balls the extra logarithmic factor has the form $\log (2n/k)$. 

\begin{Theorem}\label{IT1} Let $p\in [2,\infty)$ and $X_N\subset L_\infty(\Omega)$. Denote
$$
M_p(X_N):= \sup_{f\in X_N;f\neq 0} \|f\|_\infty/\|f\|_p.
$$
 Then for any set 
$\Omega_n =\{x^j\}_{j=1}^n\subset \Omega$ we have
 \be\label{B1}
 \e_k(X_N^p,L_\infty(\Omega_n)) \le C(p) M_p(X_N)\left(\frac{\log(2n/k)}{k}\right)^{1/p},\quad k=1,\dots,n.
 \ee
 \end{Theorem}

Lemma \ref{IL1} was proved in \cite{Kos} with a help of  Talagrand's fundamental result in functional analysis (see \cite{Tal}, p.552, Lemma 16.5.4). Theorem \ref{IT1} is based on a slight improvement of Talagrand's result. We now proceed to the main result of the paper. For a Banach space $X$ we define the modulus of smoothness
$$
\rho(u):= \rho(X,u) := \sup_{\|x\|=\|y\|=1}\left(\frac{1}{2}(\|x+uy\|+\|x-uy\|)-1\right).
$$
The uniformly smooth Banach space is the one with the property
$$
\lim_{u\to 0}\rho(u)/u =0.
$$
In this paper we only consider uniformly smooth Banach spaces with power type modulus of smoothness $\rho(u) \le \ga u^q$, $1< q\le 2$. 

Let $\D_n=\{g_j\}_{j=1}^n$ be a system of  elements of cardinality $|\D_n|=n$ in a Banach space $X$. We equip the linear space 
 $W_n:=[\D_n]:= \sp\{\D_n\}$ with the norm
 \be\label{I3}
 \|f\|_A := \|f\|_{A_1(\D_n)} := \inf\left\{ \sum_{j=1}^n |c_j|\, : \, f = \sum_{j=1}^n c_jg_j\right\}.
 \ee
 Denote by $W_{n,A}$ the $W_n$ equipped with the norm $\|\cdot\|_A$.
 We are interested in the dual norm to the norm $\|\cdot\|_A$, which we denote $\|\cdot\|_U$:
 $$
 \|F\|_U:= \|F\|_{U(\D_n)}:=   \sup_{f\in W_n; \|f\|_A\le 1} |F(f)|.
 $$
 Denote $W_{n,U}^*$ the $W_n^*$ equipped with the norm $\|\cdot\|_U$. Note that $\|\cdot\|_U$ is a semi-norm on the dual to $X$, space $X^*$. 

 \begin{Theorem}\label{IT2} Let $X$ be $q$-smooth: $\rho(X,u) \le \gamma u^q$, $1<q\le 2$ and let $\D_n$ be a  normalized system in $X$ of cardinality $|\D_n|=n$.  Then for the unit ball $B(X^*)$ of $X^*$ we have  
  \be\label{I5}
 \e_k(B(X^*),\|\cdot\|_{U(\D_n)}) \le C(X) \left(\frac{\log(2n/k)}{k}\right)^{1-1/q},\quad k=1,\dots,n.
 \ee
 \end{Theorem}
 
 \begin{Remark}\label{IR1} By Remark \ref{AR1} in the case both spaces $X$ and its dual $X^*$ satisfy conditions $\rho(X,u) \le \gamma u^q$, $\rho(X^*,u) \le \gamma u^q$ with some $\gamma >0$ and $q\in (1,2]$ we can replace in Theorem \ref{IT2} the constant $C(X)$ by $C(\gamma,q)$.
 \end{Remark}

We note that  Talagrand's above mentioned  result corresponds to Theorem \ref{IT2} with 
$\log (2n/k)$ replaced by $\log n$. We point out that the proof of Theorem \ref{IT2} (see Section \ref{A}) uses different ideas than the ones from \cite{Tal}. Our proof is based on results from greedy approximation. It is important to state that the bounds in Theorem \ref{IT2} do not allow further improvements. Indeed, let us apply Theorem \ref{IT2} in the following classical case.
Let $X=\ell^n_q$, $1<q\le 2$. Then it is known (see, for instance, \cite{DGDS}) that $\rho(\ell^n_q,u) \le u^q/q$. In the case $q\in [2,\infty)$ it is known (see, for instance, \cite{DGDS}) that $\rho(\ell^n_q,u) \le (q-1)u^q/2$. Let $\D_n$ be the canonical basis of $\ell^n_q$. Then the $\|\cdot\|_A$ is the $\ell^n_1$ norm and the $\|\cdot\|_U$ is the $\ell^n_\infty$ norm. Clearly, $X^*= \ell^n_p$, $p=q/(q-1)$. Then Theorem \ref{IT2} and Remark \ref{IR1} give for the unit ball $B^n_p$ of $\ell^n_p$, $2\le p<\infty$,
 \be\label{I6}
 \e_k(B^n_p,\ell^n_\infty) \le C(p) \left(\frac{\log(2n/k)}{k}\right)^{1/p},\quad k=1,\dots,n.
 \ee
It is known (see, for instance, \cite{DTU}, p.96, and the discussion there) that bound 
(\ref{I6}) cannot be improved for $k\in [\log n,n]$. In case $k\le \log n$ there is a trivial bound 
$\e_k(B^n_p,\ell^n_\infty)\le 1$, which is better than (\ref{I6}).

\section{Bounds for entropy numbers of octahedra} 
\label{A}

Let $\D_n=\{g_j\}_{j=1}^n$ be a system of  elements of cardinality $|\D_n|=n$ in a Banach space $X$. Consider best $m$-term approximations of $f$ with respect to $\D_n$
$$
\sigma_m(f,\D_n)_X:= \inf_{\{c_j\};\Lambda:|\Lambda|=m}\|f-\sum_{j\in \Lambda}c_jg_j\|.
$$
For a function class $W$ set
$$
\sigma_m(W,\D)_X:=\sup_{f\in W}\sigma_m(f,\D)_X.
$$
The following Theorem \ref{AT1} was proved in \cite{VT138} (see also \cite{VTbookMA}, p.331, Theorem 7.4.3).
\begin{Theorem}\label{AT1} Let a compact $W\subset X$ be such that there exists a  system $\D_n$, $|\D_n|=n$, and  a number $r>0$ such that 
$$
  \sigma_m(W,\D_n)_X \le m^{-r},\quad m\le n.
$$
Then for $k\le n$
\begin{equation}\label{A3}
\e_k(W,X) \le C(r) \left(\frac{\log(2n/k)}{k}\right)^r.
\end{equation}
\end{Theorem}

For a given set $\D_n=\{g_j\}_{j=1}^n$ of elements we introduce the octahedron (generalized octahedron)
\be\label{A4}
A_1(\D_n) := \left\{f\,:\, f=\sum_{j=1}^n c_jg_j,\quad \sum_{j=1}^n |c_j|\le 1\right\}.
\ee
Note that in the case $X = \ell^n_1$ and $g_j =e_j$, where $\{e_j\}_{j=1}^n$ is a canonical basis of $\ell^n_1$, the octahedron $A_1(\D_n)$ coincides with the regular octahedron in $\R^n$. 

The following Corollary \ref{AC1} of Theorem \ref{AT1} was obtained in \cite{VT138} (see also \cite{VTbookMA}, p.332, Corollary 7.4.7).

\begin{Corollary}\label{AC1} Let $X=L_p$, $1<p<\infty$. For a normalized system $\D_n$ of cardinality $|\D_n|=n$ we have
\begin{equation}\label{A5}
\e_k(A_1(\D_n),L_p)\le C(p)\left(\frac{\log(2n/k)}{k}\right)^{1-\max(\frac{1}{2},\frac{1}{p})},\quad k\le n.
\end{equation}
\end{Corollary} 

We need the following version of Corollary \ref{AC1}, which we prove here for completeness. 

 \begin{Theorem}\label{AT2} Let $X$ be $q$-smooth: $\rho(X,u) \le \gamma u^q$, $1<q\le 2$. Then for any  normalized system $\D_n$ of cardinality $|\D_n|=n$ we have
 \be\label{A6}
 \e_k(A_1(\cD_n),X) \le C(q,\ga) \left(\frac{\log(2n/k)}{k}\right)^{1-1/q},\quad k=1,\dots,n.
 \ee
 \end{Theorem}
 \begin{proof} Consider a new Banach space $W_{n,X}$, which is defined as $W_n:=[\D_n]:= \sp\{\D_n\}$ equipped with the norm $\|\cdot\|_X$. Then $\D_n$ is a dictionary for $W_{n,X}$. Clearly,
 $$
 \rho(W_{n,X},u) \le \rho(X,u) \le \ga u^q.
 $$
 Then it is known (see \cite{DGDS} and \cite{VTbook}, p.342, Theorem 6.8) that
 \be\label{A7}
 \sigma_m(A_1(\D_n),\D_n)_X \le C(q,\ga)m^{1/q-1},\quad m=1,2,\dots.
 \ee
 We now apply Theorem \ref{AT1} with $W=A_1(\D_n)$ and $r=1-1/q$ and complete the proof 
 of Theorem \ref{AT2}.
 
 \end{proof}
 
 We proceed to the dual version of Theorem \ref{AT2}. As above we equip the space 
 $W_n:=[\D_n]:= \sp\{\D_n\}$ with the norm
 \be\label{A8}
 \|f\|_A := \|f\|_{A_1(\D_n)} := \inf\left\{ \sum_{j=1}^n |c_j|\, : \, f = \sum_{j=1}^n c_jg_j\right\}.
 \ee
 We are interested in a dual norm to the norm $\|\cdot\|_A$, which we denote $\|\cdot\|_U$.
 For a Banach space $X$ denote $X^*$ its dual (conjugate) and for $F\in X^*$ and $f\in X$ we  write for convenience
 $$
 \<F,f\>:= \<f,F\> := F(f).
 $$
For two Banach spaces $X$, $Y$ and a bounded linear operator $A:\, X\to Y$ the dual (adjoint, conjugate) operator 
$A^*:\, Y^*\to X^*$ is the one with a property: for all $x\in X$ and all $y^* \in Y^*$ we have
$$
\<Ax,y^*\> = \<x,A^*y^*\>.
$$
We need the following simple claim. Note that the norm $\|F\|_{\D_n}$ (see (\ref{A9})) with $\D_n$ replaced by a dictionary $\D$ in $X$ is widely used in greedy approximation (see \cite{VTbook}, Ch.6).
\begin{Proposition}\label{AP1} Let $W_n$ and $\|\cdot\|_A$ be as above. Then for $F\in W_n^*$
\be\label{A9}
\|F\|_U:= \|F\|_A^* := \sup_{f\in W_n; \|f\|_A\le 1} |\<F,f\>| = \max_{1\le j\le n} |\<F,g_j\>|=:\|F\|_{\D_n}.
\ee
\end{Proposition}
\begin{proof} First, for any $f=\sum_{j=1}^n c_j g_j$  we have 
$$
|\<F,f\>| =\left|\sum_{j=1}^n c_j \<F,g_j\>\right| \le \left(\sum_{j=1}^n |c_j|\right)\left(\max_{1\le j\le n} |\<F,g_j\>|\right),
$$
which proves the inequality $\le$ in (\ref{A9}). 

Second, obviously, $\|g_j\|_A \le 1$, $j=1,\dots,n$. This proves the inequality $\ge$ in (\ref{A9}) and completes the proof of Proposition \ref{AT1}.

\end{proof}

{\bf Proof of Theorem \ref{IT2}.} We now prove the
  main result of this paper -- Theorem \ref{IT2}.  
  Let $B_V$ denote the unit ball of a Banach space $V$. For a linear operator $J:\, V\to W$ denote 
by $J(B_V):= \{w\in W:\, w=Jv,\, \|v\|_V \le 1\}$ the image of the unit ball $B_V$. Define the entropy numbers of the compact operator $J$ as follows
$$
\e_k(J) := \e_k(J, V\to W) := \e_k(J(B_V),W).
$$
We need a duality result for the entropy numbers proved in \cite{BPST}.

\begin{Theorem}\label{AT4} Let $V$ be a uniformly convex Banach space. Let $J:\, V\to W$ be a compact operator. Then for every $m\in \N$ and $p\in [1,\infty)$
$$
C_0^{-p} \sum_{k=0}^m \e_k(J^*)^p \le \sum_{k=0}^m \e_k(J)^p \le C_0^p \sum_{k=0}^m \e_k(J^*)^p,
$$
where $C_0$ depends only on $V$.
\end{Theorem}

Our assumption that $X$ is uniformly smooth implies that $X^*$ is uniformly convex (see \cite{LT}, p.61). We will apply Theorem \ref{AT4} with $V=X^*$ and $W= W_{n,U}^*$. We now define the operator $J$. Let $Id_n$ be the identity operator from $W_{n,A}$  to $X$. The dual operator $Id_n^*$ will map $X^*$ to $W_{n,U}^*$. We set $J=Id_n^*$. Then $J^* = Id_n$ and 
\be\label{A11}
\e_k(J^*) = \e_k(Id_n, W_{n,A} \to X) = \e_k(A_1(\D_n),X).
\ee
It is sufficient to prove Theorem \ref{IT2} in the case when $k$ is an even number. Take 
$m\le n$ to be an even number and set $p= q'/2$, $q':= \frac{q}{q-1}$, where $q$ is from Theorem \ref{IT2}. By Theorem \ref{AT4} we obtain
\be\label{A12}
(m/2)\e_m(J)^p \le \sum_{k=m/2}^m \e_k(J)^p\le \sum_{k=0}^m \e_k(J)^p\le  C_0^p \sum_{k=0}^m \e_k(J^*)^p.
\ee
Using Theorem \ref{AT2} and (\ref{A11}),  we continue 
\be\label{A13}
\le C_0^p\left(1+ \sum_{k=1}^m C(q,\ga)^p \left(\frac{\log(2n/k)}{k}\right)^{1/2} \right)\le 
C(X)^p m^{1/2} (\log(2n/m))^{1/2}.
\ee
Thus, we obtain from (\ref{A12}) and (\ref{A13})
\be\label{A14}
\e_m(J) \le C(X)\left(\frac{\log(2n/m)}{m}\right)^{1/q'}.
\ee
From the definition of the dual operator we have for any $f\in W_n$ and any $v\in X^*$
$$
\<f,v\>= \<J^*f,v\> = \<f,Jv\> 
$$
and, therefore, 
\be\label{A15}
\forall v\in X^*\quad  \text{we have}\quad  \|v\|_U = \|Jv\|_U.
\ee 
We now derive from (\ref{A15}) that 
\be\label{A16}
\e_m(B(X^*),\|\cdot\|_U) \le \e_m(J(B(X^*),W_{n,U}^*).
\ee
Indeed, suppose that $v\in X^*$ and $w \in W_{n,U}^*$ is such that $\|Jv-w\|_U \le \e$.
Define $u\in X^*$ as follows: for $f\in W_n$ set $\<f,u\>=\<f,w\>$ and by the Hanh-Banach theorem extend it to the whole $X$. Then, for any $f\in W_n$ we have
$$
\<f,u\> = \<J^*f,u\> = \<f,Ju\>
$$
and, therefore, $Ju=w$. By (\ref{A15}) we obtain 
$$
\|v-u\|_U = \|Jv-Ju\|_U = \|Jv-w\|_U \le \e. 
$$
This, in turn, implies  (\ref{A16}). Finally
\be\label{A17}
 \e_m(J(B_{X^*}),W_{n,U}^*)=\e_m(J).
 \ee
 A combination of (\ref{A14}), (\ref{A16}), and (\ref{A17}) implies (\ref{I5}), which completes the proof of Theorem \ref{IT2}.
 
 \begin{Remark}\label{AR1} By Remark (1) from \cite{BPST} and by the duality property between the uniform convexity of $X$ and the uniform smoothness of $X^*$ (see \cite{LT}, p.61) we obtain the following property of the constant $C_0$ in Theorem \ref{AT4}. If both $V$ and $V^*$ satisfy $\rho(V,u) \le \gamma u^q$ and $\rho(V^*,u) \le \gamma u^q$ with some 
 $\gamma>0$ and $q\in (1,2]$, then the constant $C_0$ in Theorem \ref{AT4} depends only on 
 $\gamma$ and $q$. This and the proof of Theorem \ref{IT2} imply that in the case both spaces $X$ and its dual $X^*$ satisfy the conditions $\rho(X,u) \le \gamma u^q$, $\rho(X^*,u) \le \gamma u^q$ with some $\gamma >0$ and $q\in (1,2]$ we can replace in Theorem \ref{IT2} the constant $C(X)$ by $C(\gamma,q)$.
 \end{Remark}

\section{Proof of Theorem \ref{IT1}}
\label{B}
 
We prove here Theorem \ref{IT1}.
 Let $X_N\subset L_\infty(\Omega)$ be an $N$-dimensional subspace and let a set $\Omega_n =\{x^j\}_{j=1}^n$ be a set of points from $\Omega$.  
 
  Let $\{u_i\}_{i=1}^N$ be an orthonormal basis of $X_N$. Denote the corresponding Dirichlet kernel 
$$
D_N(x,y) := \sum_{i=1}^N u_i(x)u_i(y),\quad x,y\in\Omega.
$$
We need a known technical lemma (see, for instance, \cite{VTbookMA}, p.91, Lemma 3.3.4).
We use the notation
$$
v \perp X_N\quad \text{to mean that}\quad \<v,f\> =0 \quad \forall f\in X_N.
$$

\begin{Lemma}\label{BL1} Let $p':=p/(p-1)$ be a dual exponent to $p \in [1,\infty)$. Then
$$
M_p(X_N) = \sup_{x\in \Omega} \inf_{v\in L_{p'}:\, v\perp X_N} \|D_N(x,\cdot)- v(\cdot)\|_{p'}.
$$
\end{Lemma}
\begin{proof} For each $x\in \Omega$ by the Nikol'skii duality theorem (see, for instance, \cite{VTbookMA}, p.509) we obtain
$$
\sup_{f\in X_N^{p}} |f(x)| = \sup_{f\in X_N^{p}}\left|\int_\Omega D_N(x,y) f(y) d\mu\right| =
\inf_{v\in L_{p'}:\, v\perp X_N} \|D_N(x,\cdot)- v(\cdot)\|_{p'}.
$$
It remains to observe that 
$$
M_p(X_N) = \sup_{x\in \Omega}\sup_{f\in X_N^{p}} |f(x)|,
$$
which completes the proof of Lemma \ref{BL1}.
\end{proof}

We continue the proof of Theorem \ref{IT1}. For each $j \in [1,n]$, using Lemma \ref{BL1}, we find a $v_j$ such that $v_j\perp X_N$ and
\be\label{B2}
\|D_N(x^j,\cdot)- v_j(\cdot)\|_{p'} \le 2M_p(X_N).
\ee
Denote
$$
w_j(y) := D_N(x^j,y)- v_j(y),\quad g_j(y) := w_j(y)/\|w_j\|_{p'}.
$$
We now apply results of Section \ref{A}. We set $X=L_{p'}$. Then $X^*=L_p$. Set 
$$
\D_n =\{g_j\}_{j=1}^n,\quad W_n = [\D_n].
$$
It is well known (see, for instance, \cite{DGDS}) that for $q\in [1,2]$ we have
$$
\rho(L_q,u) \le u^q/q.
$$
Applying Theorem \ref{IT2} and Remark \ref{IR1} we obtain
\be\label{B3}
 \e_k(X_N^p,\|\cdot\|_U) \le C(p) \left(\frac{\log(2n/k)}{k}\right)^{1/p},\quad k=1,\dots,n.
 \ee
 Next, for $f\in X_N$ we have
 $$
\<f,g_j\|w_j\|_{p'}\>= \<f,w_j\> = \<f(y),D_N(x^j,y)\> = f(x^j).
$$
Therefore, taking into account (\ref{B2}) we obtain for $f\in X_n$
\be\label{B4}
\|f\|_{L_\infty(\Omega_n)} \le 2M_p(X_N) \|f\|_U.
\ee
 Combining (\ref{B3}) and (\ref{B4}) we complete the proof of Theorem \ref{IT1}.

 {\bf Acknowledgement.} The work was supported by the Russian Federation Government Grant N{\textsuperscript{\underline{o}}}14.W03.31.0031. The paper contains results obtained in frames of the program \lq\lq Center for the storage and analysis of big data", supported by the Ministry of Science and High Education of Russian Federation (contract 11.12.2018 N{\textsuperscript{\underline{o}}}13/1251/2018 between the Lomonosov Moscow State University and the Fund of support of the National technological initiative projects).


\begin{thebibliography}{9999}


\bibitem{BLM} J. Bourgain, J. Lindenstrauss, and V. Milman, Approximation of zonoids by zonotopes, Acta Math., {\bf 162} (1989), 73--141.

\bibitem{BPST} J. Bourgain, A. Pajor,  S.J. Szarek, N. Tomczak-Jaegermann,  
On the duality problem for entropy numbers of operators,
In Geometric aspects of functional analysis, 1989,  50-63, 
Springer, Berlin, Heidelberg.


\bibitem{Carl} B. Carl, Entropy numbers, $s$-numbers, and eigenvalue problem, J.
Func. Analysis, {\bf 41} (1981), 290--306.

\bibitem{DPTT} F. Dai, A. Prymak, V.N. Temlyakov, and  S. Tikhonov, Integral norm discretization and related problems,
Russ. Math. Surv., {\bf 74}  (2019), 579--630.
 Translation from
Uspekhi Mat. Nauk,  {\bf 74},	Is. 4(448) (2019),	3--58; arXiv:1807.01353v1 [math.NA] 3 Jul 2018.

\bibitem{DPSTT1}
F. Dai, A. Prymak, A. Shadrin, V. Temlyakov, S. Tikhonov,
Sampling discretization of integral norms,
arXiv:2001.09320v1 [math.CA] 25 Jan 2020.
 
 \bibitem {DPSTT2} F. Dai, A. Prymak, A. Shadrin, V. Temlyakov, and S. Tikhonov, Entropy numbers and Marcinkiewicz-type deiscretization theorems, arXiv:2001.10636v1 [math.CA] 28 Jan 2020.


\bibitem{DTU} Ding D{\~u}ng, V.N. Temlyakov, and T. Ullrich, Hyperbolic Cross Approximation, Advanced Courses in Mathematics CRM Barcelona, Birkh{\"a}user, 2018; arXiv:1601.03978v2 [math.NA] 2 Dec 2016.

\bibitem{DGDS}  M. Donahue, L. Gurvits, C. Darken, E. Sontag,  Rate of convex approximation in non-Hilbert spaces,  Constr. Approx., {\bf 13 } (1997),  187--220.

\bibitem{GZ} E. Gine and J. Zinn, Some limit theorems for empirical processes, 
Ann. Prob., {\bf 12} (1984), 929--989.

 \bibitem{HPV} A. Hinrichs, J. Prochno, and J. Vybiral, Entropy numbers of embeddings of Schatten classes, J. Functional Analysis, {\bf 273} (2017), 3241--3261; arXiv:1612.08105v1 [math.FA] 23 Dec 2016.

\bibitem{KoTe} S.V. Konyagin and V.N. Temlyakov, The entropy in learning theory. Error estimates, Constr. Approx., {\bf 25} (2007), 1--27.

\bibitem{Kos} E. Kosov, The Marcinkiewicz-type discretization
of $L^p$-norms under the Nikolskii-type assumptions, arXiv:2005.01674v1 [math.FA] 4 May 2020.

\bibitem{LT} J. Lindenstrauss and L. Tzafriri, Classical Banach Spaces II, Springer-Verlag, Berlin, 1979.

\bibitem{LGM} G. Lorentz, M. von Golitschek, and Y. Makovoz, Constructive Approximation: Advanced Problems. {S}pringer, Berlin, 1996.

\bibitem{MSS} A. Marcus, D.A. Spielman, and N. Srivastava,
Interlacing families II: Mixed characteristic polynomials and the Kadison-Singer problem, Annals of Math., {\bf 182} (2015), 327--350.

\bibitem{Schu} C. Sch{\" u}tt, Entropy numbers of diagonal operators between
symmetric Banach spaces,
 J. Approx. Theory, {\bf 40} (1984), 121--128.
 
 \bibitem{Tal} M. Talagrand, Upper and lower bounds for stochastic processes: modern methods and classical problems.
-- Springer Science and Business Media, 2014.

\bibitem{VT138} V.N. Temlyakov, An inequality for the entropy numbers and its application,
J. Approx. Theory, {\bf 173} (2013), 110--121.

\bibitem{VTbook} V.N. Temlyakov, Greedy Approximation, Cambridge University
Press, 2011.

\bibitem{VT156} V.N. Temlyakov, On the entropy numbers of the mixed smoothness function classes, J. Approx. Theory, {\bf 207} (2017), 26--56; arXiv:1602.08712v1 [math.NA] 28 Feb 2016.

\bibitem{VT158} V.N. Temlyakov, The Marcinkewiecz-type discretization theorems for the hyperbolic cross polynomials, Jaen  Journal on Approximation, {\bf 9} (2017), No. 1, 37--63; arXiv: 1702.01617v2 [math.NA] 26 May 2017.

\bibitem{VT159} V.N. Temlyakov, The Marcinkiewicz-type discretization theorems, Constr. Approx. {\bf 48} (2018), 337--369; arXiv: 1703.03743v1 [math.NA] 10 Mar 2017.

\bibitem{VTbookMA} V. Temlyakov, Multivariate Approximation, Cambridge University Press, 2018.

  
\end{thebibliography}
\end{document}